\documentclass[a4paper,leqno,11pt]{amsart}

\usepackage{amsfonts,amssymb,verbatim,amsmath,amsthm,latexsym,textcomp,amscd}
\usepackage{latexsym,amsfonts,amssymb,epsfig,verbatim,extarrows}
\usepackage{amsmath,amsthm,amssymb,latexsym,graphics,textcomp}
\usepackage{xspace}
\usepackage{graphicx}
\usepackage{color}
\usepackage{url}
\usepackage{enumerate}
\usepackage[mathscr]{euscript}
\usepackage[all, cmtip]{xy} 
\usepackage{hyperref}
\hypersetup{
	colorlinks,
	citecolor=red,
	filecolor=black,
	linkcolor=blue,
	urlcolor=black
}

\DeclareMathAlphabet{\mathpzc}{OT1}{pzc}{m}{it}

\theoremstyle{plain}
\newtheorem{theorem}{Theorem}[section]
\newtheorem{prop}[theorem]{Proposition}
\newtheorem{lemma}[theorem]{Lemma}
\newtheorem{cor}[theorem]{Corollary}
\newtheorem{defn}[theorem]{Definition}

\newtheorem{exam}[theorem]{Example}

    \newtheorem{notation}[theorem]{Notation}
    \newtheorem{subsec}[theorem]{}

\newtheorem*{thma}{Theorem A}
\newtheorem*{thmb}{Theorem B}

\newenvironment{myeq}[1][]
{\stepcounter{theorem}\begin{equation}\tag{\thetheorem}{#1}}
{\end{equation}}

\newenvironment{mysubsection}[2][]
{\begin{subsec}\begin{upshape}\begin{bfseries}{#2.}
\end{bfseries}{#1}}
{\end{upshape}\end{subsec}}

\newcommand{\Z}{{\mathbb{Z}}}

\newcommand{\Q}{{\mathbb Q}}

\newcommand{\Hom}{\mathit{Hom}}

\newcommand\A{{\mathcal A}}

\newcommand\FF{{\mathcal F}}

\newcommand\LL{{\mathcal L}}
\newcommand\MM{{\mathcal M}}

\newcommand\PP{{\mathcal P}}

\newcommand\RR{{\mathcal R}}

\newcommand\PMF{{\PP\kern-2pt\MM\FF}}

\newcommand\PML{{\PP\kern-2pt\MM\LL}}

\newcommand{\fsubd}{\mathrel{{\scriptstyle\searrow}\kern-1ex^d\kern0.5ex}}
\newcommand{\bsubd}{\mathrel{{\scriptstyle\swarrow}\kern-1.6ex^d\kern0.8ex}}
\newcommand{\fsubeq}{\mathrel{\raise-.7ex\hbox{$\overset{\searrow}{=}$}}}
\newcommand{\bsubeq}{\mathrel{\raise-.7ex\hbox{$\overset{\swarrow}{=}$}}}

\newcommand{\tsh}[1]{\left\{\kern-.9ex\left\{#1\right\}\kern-.9ex\right\}}

\newcommand{\Lie}{\mathit{Lie}}

\newcommand{\Ker}{\mathit{Ker}}

\begin{document}

\title[]{The homotopy type of the loops on $(n-1)$-connected $(2n+1)$-manifolds}

\author{Samik Basu}
\email{samik.basu2@gmail.com; samikbasu@isical.ac.in}
\address{Stat-Math Unit, Indian Statistical Institute, Kolkata - 700108, India.}

 


\maketitle

\begin{abstract}
For $n\geq 2$ we compute the homotopy groups of $(n-1)$-connected closed manifolds of dimension $(2n+1)$. Away from the finite set of primes dividing the order of the torsion subgroup in homology, the $p$-local homotopy groups of $M$ are determined by the rank of the free Abelian part of the homology.  Moreover, we show that these $p$-local homotopy groups can be expressed as a direct sum of $p$-local homotopy groups of spheres. The integral homotopy type of the loop space is also computed and shown to depend only on the rank of the free Abelian part and the torsion subgroup. 
\end{abstract}


\setcounter{section}{0}

\section{Introduction}\label{intro}

The computation of homotopy groups of topological spaces is an important problem in topology and also in the solution of many problems in differential topology such as cobordism, and surgery theory. However, even for simple examples like spheres the computations are very hard and mostly unknown. Serre \cite{Ser51} proved that the homotopy groups of spheres are finitely generated, and $\pi_i(S^n)$ is torsion except in the cases $\pi_n(S^n)\cong \Z$ and $\pi_{4n-1}(S^{2n})$ which is isomorphic to $\Z$ direct sum a finite group.  Toda \cite{Tod62} made extensive computations with the homotopy groups of spheres which are still among the best calculations of these groups. A more systematic calculation has been carried out for the stable homotopy groups $\pi_n^s$ which are isomorphic to $\pi_{k+n}S^k$ for $k> n+1$ \cite{Rav86}.

Results about the homotopy groups of manifolds and associated CW complexes and more generally, computations in unstable homotopy theory have been of widespread interest.  Serre \cite{Ser51} showed that a simply connected, finite CW complex $X$ has infinitely many non-zero homotopy groups, and conjectured that such a space with non-trivial $\Z/p\Z$ cohomology has the property that $\pi_n(X)$ contains $\Z/p\Z$ for infinitely many values of $n$. This was proved later by McGibbon and Neisendorfer \cite{McGN85} as an application of Miller's results on the Sullivan conjecture \cite{Mill84}.  

A curious result was observed by James \cite{Jam57Ann} that $2^{2n}$ annihilates the $2$-torsion in $\pi_qS^{2n+1}$ for any $q$.  For odd primes, the analogous result that $p^{2n}$ annihilates the $p$-torsion in $\pi_qS^{2n+1}$, was proved by Toda \cite{Tod56}. These factors are called homotopy exponents of spheres at the prime $p$, and one defines them for any space $X$.  The best possible exponents for spheres at odd primes were conjectured to be of half the order as above by Barratt. This was first proved by Selick \cite{Sel78} for $S^3$ and by Cohen, Moore and Neisendorfer \cite{CMN79, CMN79b,Nei81} for other spheres.  

Apart from spheres, homotopy groups of aspherical manifolds are easy to compute being $K(\pi,1)$-spaces, and homotopy groups of projective spaces are computable in terms of the homotopy groups of spheres. Hilton \cite{Hil55} computed the homotopy groups of a wedge of spheres demonstrating them as a direct sum of homotopy groups of spheres which are mapped onto the wedge by Whitehead products.  Milnor \cite{Ada72} generalized this to the loop space of suspension of a wedge of spaces, a result known as the Hilton-Milnor Theorem. 

The spaces next in line in terms of complexity of cell structures are those obtained by attaching a cell to a wedge of spheres. Primary examples of these are simply connected $4$-manifolds, and more generally $(n-1)$-connected $2n$-manifolds. Under suitable torsion free assumptions $(n-1)$-connected $(2n+1)$-manifolds and even more generally $(n-1)$-connected $d$-manifolds with $d\leq 3n-2$ are also of this type. In recent times there have been a number computations for these manifolds \cite{BeTh14, BaBa15, BaBa16_unpub} which we recall now. 

A simply connected $4$-manifold has a cell structure with a single $4$-cell attached to a wedge of $r$ copies of $S^2$. Curiously if $r\geq 2$ there is a circle bundle over this $4$-manifold whose total space is a connected sum of $(r-1)$ copies of $S^2\times S^3$ \cite{BaBa15, DuLi05}. The methods involve a simple geometric argument followed by an application of Smale's classification of simply connected spin $5$-manifolds \cite{Sma62}. This implies as a corollary that the formula for the homotopy groups of a simply connected $4$-manifold depend only on the middle Betti number. The circle bundle is used in \cite{BaBa15} to compute rational homotopy groups of such a $4$-manifold. 

A natural question regarding the above calculation of the homotopy groups of simply connected $4$-manifolds was that whether such results could be proved independently of geometric results such as Smale's classification. There are now two different solutions of the above problem in \cite{BeTh14} (which follows an idea that first appeared in \cite{BeWu15})  and \cite{BaBa15_unpub}. Both these papers analyze the homotopy type of the loop space of a $(n-1)$-connected $2n$-manifold and prove certain loop space decompositions. Beben and Theriault \cite{BeTh14} prove a decomposition of the loop space by proving a more general result on manifolds $P$ with a certain cofibre $Q$  whose homology resembles $S^m \times S^{m-n}$. Under certain torsion-free and multiplicative conditions on the cohomology of $P$, the loop space of $P$ splits as a product of $\Omega Q$ and $\Omega F$, where $F$ is the homotopy fibre of the map $P\to Q$. The $(n-1)$-connected $2n$-manifolds for $n\neq 2,4,8$ satisfied the hypothesis of the general result and so this implied a  decomposition of  the loop space into simpler factors. For simply connected $4$-manifolds a slight modification of the argument was used to prove the decomposition of the loop space. This kind of argument first appeared in \cite{BeWu15}, where Beben and Wu  have considered $(n-2)$-connected $2n-1$ Poincar\'{e} Duality complexes for $n$ even and by analyzing the attaching map of the top cell, have obtained an analogous loop space decomposition result.  The paper \cite{BeWu15} which appeared as a preprint in 2011, was the first paper which considered loop space decompositions of highly connected manifolds. Fred Cohen has commented that generally the homotopy type of loop spaces of manifolds is quite a complicated subject.  

The second approach to the loop spaces of highly connected manifolds is  \cite{BaBa15_unpub}, where we have canonically associated to a $(n-1)$-connected $2n$-manifold $M$, whose middle Betti number is $\geq 2$, a quadratic Lie algebra which is torsion-free. We consider a suitable basis for this Lie algebra and write down maps from loop spaces of spheres mapping into $\Omega M$ corresponding to the basis elements. These are used to prove that $\Omega M$ is a weak product (homotopy colimit of finite products) of loop spaces of spheres. Consequently, there is an expression of the homotopy groups of such $M$ as a direct sum of the homotopy groups of spheres. If the Betti number is $1$, $n$ is forced to be $2$, $4$ or $8$ by the Hopf invariant one problem \cite{Ada60}, and in this case we have observed that an analogous result is true only after inverting finitely many primes. These arguments have also been carried out for $(n-1)$-connected $d$-manifolds with $d\leq 3n-2$ after inverting finitely many primes \cite{BaBa16_unpub}.

In this paper, we compute homotopy groups of $(n-1)$-connected $(2n+1)$-manifolds in terms of homotopy groups of spheres. These results complement \cite[Theorem 1.1, Theorem 6.4]{BeWu15} in the case $n$ is odd, and are new for $n$ even. From Poincar\'{e} duality one proves that the homology of such a manifold $M$ is described by 
$$ 
H_i(M)=\begin{cases} 
                \Z &\mbox{if}~i=0,2n+1 \\
                \Z^r \oplus G &\mbox{if}~i=n  \\
               \Z^r                &\mbox{if} ~i=n+1           \\
0    &\mbox{otherwise},                   
\end{cases}
$$
where $G$ is a finite Abelian group. We make two kinds of computations for such manifolds. The first is an expression of the homotopy groups of $M$ as a direct sum of homotopy groups of spheres in the case $G=0$, or when $G$ is non-trivial, after inverting all the primes dividing the order of $G$ along the lines of \cite{BaBa15_unpub}. We prove the following result (see Theorem \ref{htpytors-free}, Theorem \ref{htpyformtors-free})
\begin{thma}\label{A}
Let $M$ be a $(n-1)$-connected $(2n+1)$-manifold with notations as above, and satisfying $r\geq 1$. Let $p$ be a prime such that $p\nmid |G|$. Then, \\
(a) The $p$-local homotopy groups of $M$ can be expressed as a direct sum of $p$-local homotopy groups of spheres. \\
(b) The $p$-local homotopy groups of $M$ are a function of $r$ and do not depend on the attaching map of the $(2n+1)$-cell. 
\end{thma}
There is an explicit expression for calculating the number of $\pi_k(S^l)_{(p)}$ in $\pi_k(M)_{(p)}$ in Theorem \ref{htpyformtors-free}. This expression is quite similar to \cite[Theorem 3.6, Theorem 3.7]{BaBa16_unpub} where the computation is carried out in the general case of $(n-1)$-connected $d$-manifolds with $d\leq 3n-2$. A closer inspection shows that the results in Theorem \ref{htpytors-free} and Theorem \ref{htpyformtors-free} are stronger for $(n-1)$-connected $(2n+1)$-manifolds. For, in \cite{BaBa16_unpub} the result about the number of primes being inverted in the expression above is not determined from its homology, while in the current paper we need to invert only those primes which appear as orders of elements in the torsion subgroup $G$. 

The second computation for $(n-1)$-connected $(2n+1)$-manifolds involves a decomposition of the loop space into simpler factors along the lines of \cite{BeTh14}, where the torsion-free assumption is not necessary. We prove (see Theorem \ref{loopdecM}) 
\begin{thmb}
Suppose $M$ is a $(n-1)$-connected $(2n+1)$-manifold with notations as above, and satisfying $r\geq 1$. Then we have a homotopy equivalence
$$\Omega M \simeq \Omega S^n \times \Omega S^{n+1} \times \Omega (Z\vee (Z\wedge \Omega (S^n \times S^{n+1})))$$
where $Z \simeq \vee_{r-1}S^n \vee_{r-1}S^{n+1} \vee M(G,n)$. 
\end{thmb}   
In the expression above, $M(G,n)$ refers to the Moore space for $G$ of degree $n$ described by 
$$\tilde{H}_\ast M = \begin{cases} G &\mbox{if} ~ \ast =n \\ 
                                         0 &\mbox{if} ~ \ast \neq n. \end{cases} $$ 
It is instructive to compare the expression above with \cite[Theorem 6.4]{BeWu15} where it is assumed that the spaces are localized at an odd prime $p$. In the case $G=0$, the expressions match exactly while for $G\neq 0$, the expressions differ slightly. Further \cite[Theorem 1.1]{BeWu15} also includes some computations in the case $r=0$ where it is assumed that $G$ does not have any $2$ torsion. A significant point to note here is that the calculations in Theorem A and Theorem B rely only on the expressions of the cohomology algebras, and thus they carry forward for simply connected finite Poincar\'{e} duality complexes of the type above (that is for $(n-1)$-connected Poincar\'{e} duality complexes of dimension $2n+1$). 

As an application we try to compute the homotopy exponents of such manifolds. Moore's conjecture \cite{NeSe81} states that a finite complex has homotopy exponents at every prime if and only if it is rationally ellliptic. It is easily observed that the $(n-1)$-connected $(2n+1)$-manifolds are rationally hyperbolic if and only if $r\geq 2$. So, when $r\geq 2$, the Moore conjecture predicts that there will be no homotopy exponent at some prime. From Theorem A above we deduce that if $p$ does not divide the order of $G$, the manifold does not have a homotopy exponent at $p$ when $r\geq 2$. Further from Theorem B, we can deduce that the homotopy groups of $S^n \vee S^{n+1}$ are a summand of the homotopy groups of $M$ (see Corollary \ref{rethyp}). It follows that for any prime $p$, the $(n-1)$-connected $(2n+1)$-manifolds with $r\geq 2$ do not have any homotopy exponent at $p$.  

\begin{notation}
All manifolds considered in this paper are compact, closed and oriented unless otherwise mentioned.  
\end{notation}

\begin{mysubsection} {Organisation of the paper} 
In Section \ref{quadassLie}, we introduce some preliminaries on quadratic associative algebras, Lie algebras, Koszul duality, Diamond lemma and Poincar\'{e}-Birkhoff-Witt Theorems. In Section \ref{quadcoh}, we derive some quadratic properties of cohomology algebras arising from highly connected manifolds. In  Section \ref{conn}, we compute the loop space homology of a $(n-1)$-connected $(2n+1)$-manifold and use this to compute the homotopy groups as a direct sum of homotopy groups of spheres. In Section \ref{loopdectors}, we prove a decomposition of the loop space of a $(n-1)$-connected $(2n+1)$-manifold in the case that the rank of the middle homology is at least $1$.  
\end{mysubsection} 

\begin{mysubsection}{Acknowledgements}
The author would like to thank the referee for pointing out the reference \cite{BeWu15}, and also for pointing out the history of the problem of loop space decompositions of highly connected manifolds. 
\end{mysubsection}

\section{Quadratic Associative Algebras and Lie Algebras}\label{quadassLie}

In this section we introduce some algebraic preliminaries related to Koszul duality of {\it associative algebras} and associated {\it Lie algebras}. We also recall some results and relations between Gr\"obner bases of quadratic algebras, quadratic Lie algebras and, Poincar\'e-Birkhoff-Witt Theorem. These are accompanied by some crucial algebraic results used throughout the manuscript. 

\begin{mysubsection} { Koszul duality of associative algebras}
We begin with some background on Koszul duality of quadratic algebras over a field following \cite{PolPos05} and \cite{LoVa12}. Throughout this subsection $k$ denotes a field, $V$ a $k$-vector space and $\otimes=\otimes_k$ unless otherwise mentioned. 

\begin{defn}
Let $T_k(V)$ denote the tensor algebra on the space $V$. For $R\subset V\otimes_k V$, the associative algebra $A_k(V,R)=T_k(V)/(R)$ is called a {\it quadratic $k$-algebra}. 
\end{defn} 


A quadratic algebra is graded by weight. The tensor algebra $T(V)$ is {\it weight-graded} by declaring an element of $V^{\otimes n}$ to have grading $n$. Since $R$ is homogeneous, the weight grading passes onto $A(V,R):=A_k(V,R)=T_k(V)/(R)$. 

The dual notion leads to quadratic coalgebras. For this note that $T(V)$ has a coalgebra structure by declaring the elements of $V$ to be primitive. When $T(V)$ is thought of as a coalgebra we write it as $T^c(V)$. 
\begin{defn}
For $R\subset V\otimes V$ the {\it quadratic coalgebra} $C(V,R)$ is defined as the maximal sub-coalgebra $C$ of $T^c(V)$ such that $C\to T^c(V) \to V\otimes V/R$ is $0$. The maximal property of $C(V,R)$ implies that if $C$ is weight-graded sub-coalgebra such that the weight $2$ elements are contained in $R$ then $C\subset C(V,R)$. 
\end{defn}

Recall that a $k$-algebra $A$ is {\it augmented} if there is a $k$-algebra map $A\to k$. Analogously a $k$-coalgebra $C$ is {\it coaugmented} if there is a $k$-coalgebra map $k\to C$. For a coaugmented coalgebra $C$ one may write $C \cong k \oplus \bar{C}$ and the projection of $\Delta$ onto $\bar{C}$ as 
$\bar{\Delta} : \bar{C} \to \bar{C} \otimes \bar{C}.$
A coaugmented coalgebra is said to be {\it conilpotent} if for every $c\in \bar{C}$ there exists $r>0$ such that $\bar{\Delta}^r(c)=0$.  

One has adjoint functors between augmented algebras and coaugmented colgebras given by the bar construction and the cobar construction (see \cite{LoVa12}, \S 2.2.8).


\begin{defn}
Let $\bar{A}\subset A$ be the kernel of the augmentation, define $BA= (T(s\bar{A}),d)$, where $s$ denotes suspension and $d$ is generated as a coderivation by 
$$d(s(a))=s(a\otimes a)-s(a\otimes 1) -s(1\otimes a).$$
Dually, let $C=\bar{C}\oplus k$, and define $\Omega C = (T(s^{-1}\bar{C}),d)$ where $d$ is generated as a derivation by the equation 
$$d(s^{-1} c) = s^{-1}(\bar{\Delta} (c))=s^{-1}(\Delta(c)-c\otimes 1 - 1\otimes c).$$
\end{defn}

Note that the above definition makes sense for quadratic algebras (and coalgebras) as these are naturally augmented (respectively coaugmented). There is a differential on $C\otimes \Omega C$ generated by $d(c)=1\otimes s^{-1}c$ and dually a differential on $A\otimes BA$.  
\begin{defn}
The Koszul dual coalgebra of a quadratic algebra $A(V,R)$ is defined as $A^\text{!`}=C(s(V),s^2(R))$. The Koszul dual algebra $A^!$ of a quadratic algebra $A(V,R)$ is defined as $A^!=A(V^*,R^\perp)$ where $R^\perp\subset V^*\otimes V^*$ consists of elements which take the value $0$ on $R\subset V\otimes V$. 
\end{defn}

 The Koszul dual algebra and the Koszul dual coalgebra are linear dual up to a suspension. Let $A^{(n)}$ stand for the subspace of homogeneous $n$-fold products. Then $(A^!)^{(n)}\cong s^n((A^\text{!`})^*)^{(n)}$. 

 For a quadratic algebra $A(V,R)$, there is a natural map from $\Omega A^\text{!`} \rightarrow A$ which maps $v$ to itself. Using this map there is a differential on $A^\text{!`}\otimes_\kappa A$ denoted by $d_\kappa$.      

\begin{defn}{ \cite[Theorem 3.4.6]{LoVa12}}
\,A quadratic algebra $A(V,R)$ is called {\it Koszul} if one of the following equivalent conditions hold:\\
(i) $\Omega A^\text{!`}\rightarrow A$ is a quasi-isomorphism; \\
(ii) the chain complex $A^\text{!`}\otimes_\kappa A$ is acyclic; \\
(iii)\footnote{See  \cite{PolPos05}, Chapter 2, Definition 1.} $Ext_A(k,k)\cong A^!$. 
\end{defn}

Koszulness is an important property of quadratic algebras. One may make an analogous definition for quadratic coalgebras \cite[Theorem 3.4.6]{LoVa12}. From \cite[Proposition 3.4.8]{LoVa12} one has that $A(V,R)$ is Koszul if and only if $A(V^\ast, R^\perp)$ is Koszul. One may use the Koszul property to compute the homology of the cobar construction in various examples. We recall a condition for Koszulness. Fix a basis $(v_1,v_2,\ldots, v_n)$ of $V$, and fix an order $v_1<v_2<\ldots<v_n$. This induces a lexicographic order on the degree $2$ monomials. Now arrange the expressions in $R=\textup{span}_k\{r_1,r_2,\ldots\}$ in terms of order of monomials. An element $v_iv_j$ is called a leading monomial if there exists $r_l=v_iv_j + \mbox{lower order terms}$. Note that (cf. \cite{LoVa12}, Theorem 4.1.1) implies that if there is only one leading monomial $v_iv_j$ with $i\neq j$ then the algebra is Koszul. This leads to the following result.
\begin{prop}\label{Kos2}
Let $V$ be a $k$-vector space and $R=k\mathpzc{r}\subset V\otimes V$ be a $1$-dimensional subspace such that with respect to some basis $\{v_1,\ldots,v_n\}$ of $V$, 
$$\mathpzc{r}=v_iv_j +\sum_{k<i~or~k=i,l<j} a_{k,l} v_kv_l$$
for $i\neq j$ in the sense above. Then the algebra $A=A(V,R)=T(V)/R$ is Koszul.
\end{prop} 
\end{mysubsection}

\begin{mysubsection} {Lie algebras and quadratic algebras}
Let $\mathcal{R}$ be a Principal Ideal Domain (PID). We recall some facts about quadratic algebras and Lie algebras over $\mathcal{R}$. In this paper the domains used will be $\Z$ or some localization of $\Z$. Suppose $V$ is a free $\mathcal{R}$-module that is finitely generated and $R\subset V\otimes_\RR V$ be a submodule. As before, denote by $A(V,R)$ the quadratic $\RR$-algebra $T_\RR(V)/(R)$ where $T_\RR(V)$ is the $\RR$-algebra tensor algebra on $V$ and $(R)$ is the two-sided ideal on $R$.  

Recall the Diamond Lemma from \cite{Berg78}. Suppose that $V$ has a basis $x_1,\ldots,x_m$. Suppose that the submodule $R\subset V\otimes V$ is generated by $n$ relations of the form $W_i - f_i$ where $W_i=x_{\alpha(i)} \otimes x_{\beta(i)}$ and $f_i$ a linear combination of terms $x_j\otimes x_l$ other than $x_{\alpha(i)}\otimes x_{\beta(i)}$. We call a monomial $x_{i_1}\otimes \ldots \otimes x_{i_l}$ {\it $R$-irreducible} if it cannot be expressed as $A\otimes x_{\alpha(i)}\otimes x_{\beta(i)} \otimes B$ for monomials $A,B$ for any $i$. Theorem 1.2 of \cite{Berg78} states that under certain conditions on $R$, the images of the $R$-irreducible monomials in $A(V,R)$ forms a basis. One readily observes that the conditions are satisfied if $n=1$ and $\alpha(1) \neq \beta(1)$.  Therefore we conclude the following theorem.
\begin{theorem}\label{Diamond}
Suppose that $R=k\mathpzc{r}$ and $\mathpzc{r}$ is of the form 
$$x_\alpha\otimes x_\beta - \sum_{(i,j)\neq (\alpha, \beta)} a_{i,j} x_i \otimes x_j$$ 
with $\alpha\neq \beta$. Then the $R$-irreducible elements form a basis for $A(V,R)$. 
\end{theorem}

Note that if $R$ is as in Proposition \ref{Kos2}, the hypothesis of Theorem \ref{Diamond} is automatically satisfied. Thus we have 

\begin{prop}\label{freemod}
Consider a quadratic algebra $A(V,R)$ over a Principal Ideal Domain $\RR$ such that $R\subset V\otimes_\RR V$ is such that with respect to some basis $\{v_1,\ldots,v_n\}$ of $V$, 
$$\mathpzc{r}=v_iv_j +\sum_{k<i~or~k=i,l<j} a_{k,l} v_kv_l$$
for $i\neq j$ for some ordering on $\{1,\cdots, n\}$.  Then, 
\begin{enumerate}
\item As a $\RR$-module the quadratic algebra $A(V,R)$ is free. 
\item The $R$-irreducible elements form a basis. 
\item For an irreducible element $\pi\in \RR$, (so that $k=\RR/\pi$ is a field), $A_\RR(V,R)\otimes_\RR k  \cong A_k(V\otimes_\RR k,R\otimes_\RR k).$
\end{enumerate}
\end{prop}
\begin{proof}
The first two statements directly follows from the variant of Diamond Lemma presented in Theorem \ref{Diamond}. For (3), note that the construction of the quadratic algebras induces a map 
$$
A_\RR(V,R)\otimes_\RR k  \to A_k(V\otimes_\RR k,R\otimes_\RR k).
$$
Since the $R$-irreducible elements on both sides form a basis, it follows that the map is an isomorphism as the basis of the left hand side  is mapped bijectively onto the basis of the right hand side.
\end{proof} 

\mbox{ }

Note that the free Lie algebra on $V$ is contained in its universal enveloping algebra $T(V)$. This enables one to make an analogous construction of a Lie algebra in the case $R\subset V\otimes V$ lies in the free Lie algebra generated by $V$. Construct the quadratic Lie algebra $L(V,R)$ as $L(V,R)=\frac{\Lie(V)}{(R)}$, where $\Lie(V)$ is the free Lie algebra on $V$ and $(R)$ is the Lie algebra ideal generated by $R$. We readily prove the following useful result. 


\begin{prop}\label{univlie}
The universal enveloping algebra of $L(V,R)$ is $A(V,R)$. 
\end{prop}

\begin{proof}
This follows from the universal property of universal enveloping algebras. One knows that the universal enveloping algebra of the free Lie algebra $L(V)=L(V,0)$ is the tensor algebra $T_\RR(V)=A(V,0)$. Consider the composite 
$$L(V)\to T_\RR(V)\to T_\RR(V)/(R).$$
The composite is a map of Lie algebras and it maps $R$ to $0$. Hence we obtain a map $L(V,R)= \Lie(V)/(R) \to A(V,R)$.

Now suppose we have a map $L(V,R)\to A$ where $A$ is an associative algebra. This means we have an arbitrary map from $V\to A$ so that $R$ is mapped to $0$. Hence it induces a unique map from $A(V,R)\to A$. Thus $A(V,R)$ satisfies the universal property of the universal enveloping algebra.   
\end{proof} 

Recall the Poincar\'e-Birkhoff-Witt Theorem for universal enveloping algebras of Lie algebras over a commutative ring $\RR$. For a Lie algebra $\LL$ over $\RR$ which is free as a $\RR$-module, the universal enveloping algebra of $\LL$ has a basis given by monomials on the basis elements of $\LL$. We may use this to deduce the following result.

\begin{prop}
\label{PBW-quad}
Suppose $V$ and $R$ are as in Proposition \ref{freemod}. Then the quadratic algebra $A(V,R)$ is isomorphic to the symmetric algebra on the $\RR$-module $L(V,R)$.  In terms of the multiplicative structure, the symmetric algebra on $L(V,R)$ is isomorphic to the associated graded of $A(V,R)$ with respect to the length filtration induced on the universal enveloping algebra.  
\end{prop}

In the above Proposition the length filtration on $U(L)$ for a Lie algebra $L$ is defined inductively by 
$$F_0 U (L) = \RR$$
and
$$F_nU (L) = F_{n-1}U (L) + \mathit{Im}\{ L \otimes F_{n-1}U (L) \to U (L) \otimes U (L) \to U (L)\}.$$
\end{mysubsection}

\begin{mysubsection} {Bases of quadratic Lie algebras} 
For quadratic Lie algebras $L(V,R)$ over a PID $\RR$ as above, one may write down an explicit basis using methods similar to the Diamond Lemma. 

Fix $V$, a free $\RR$-module with basis $(a_1, a_2, \cdots, a_n)$, and $R$ an anti-symmetric element of $V\otimes V$ of the form 
$$a_1\otimes a_2 - a_2\otimes a_1 +  \mathit{terms~not~involving~}a_1\mathit{~and~}a_2.$$
The anti-symmetry condition ensures that $R$ can be expressed as a linear combination of $(a_i\otimes a_j - a_j\otimes a_i)$'s and hence  lies in the free Lie algebra generated by $V$. The proof of Proposition \ref{freemod} shows that for such $V$ and $R$ one obtains a basis of $A(V,R)$ as prescribed in the Diamond Lemma.

Consider the quadratic Lie algebra $L(V,R)$ for $R$ as the above. As a Lie algebra element $R$ has the form  
$$[a_1, a_2] +  \mathit{terms~not~involving~}a_1\mathit{~and~}a_2.$$

\begin{prop}
The $\RR$-module $L(V,R)$ is free.
\end{prop}

\begin{proof}
From Proposition \ref{univlie} we obtain that the enveloping algebra of $L(V,R)$ is $A(V,R)$ which is free by Proposition \ref{freemod}. Over a Dedekind Domain the map from the Lie algebra to its universal enveloping algebra is an inclusion (\cite{Car58},\cite{Laz54}; also see \cite{Cohn63}). It follows that the map $L(V,R)$ to $A(V,R)$ is an inclusion.

 Introduce the weight grading on $L(V,R)$ and $A(V,R)$ by declaring the grading of $V$ to be $1$. Then the inclusion of $L(V,R)$ in $A(V,R)$ is a graded map. Now each graded piece of $A(V,R)$ is free and finitely generated and the corresponding homogeneous piece of $L(V,R)$ is a submodule. This is free as $\RR$ is a PID.   
\end{proof}

Denote by $L^{w}(V,R)$ (respectively $A^w(V,R)$) the homogeneous elements of $L(V,R)$ (respectively $A(V,R)$) of weight $w$. Recall from \cite{LaRam95} and \cite{Loth97} the concept of a Lyndon basis of a Lie algebra described by generators and relations.

\begin{defn}
Define an order on $V$ by $a_1<a_2<\ldots <a_r$. A {\it word} (in the ordered alphabet $\{a_1,\ldots,a_r\}$) is a monomial in the $a_i$'s. The number of alphabets in a word $w$ is denoted by $|w|$ and called the {\it length} of the word. The lexicographic ordering on the words will be denoted by $>$, i.e., $w> \mathpzc{w}$ means that the word $w$ is lexicographically bigger than $\mathpzc{w}$.

A {\it Lyndon word} is one which is lexicographically smaller than its cyclic rearrangements. Fix an order on all words by declaring words $w\, \preceq\, \mathpzc{w}$ if $|w|\leq |\mathpzc{w}|$, and if $|w|=|\mathpzc{w}|$ then $w\geq \mathpzc{w}$ in lexicographic order. 
\end{defn}

Let $L$ be the set of Lyndon words. Any Lyndon word $\mathpzc{l}$ with $|\mathpzc{l}|>1$ can be uniquely decomposed into $\mathpzc{l}=\mathpzc{l}_1\mathpzc{l}_2$ where $\mathpzc{l}_1$ and $\mathpzc{l}_2$ are Lyndon words so that $\mathpzc{l}_2$ is the larget proper Lyndon word occuring from the right in $\mathpzc{l}$. Each Lyndon word $\mathpzc{l}$ is associated to an element $b(\mathpzc{l})$ of the free Lie algebra on $V$. This is done inductively by setting $b(a_i)=a_i$ and for the decomposition above 
$$b(\mathpzc{l})=b(\mathpzc{l}_1\mathpzc{l}_2):=[b(\mathpzc{l}_1),b(\mathpzc{l}_2)].$$
The image, under $b$, of the set of Lyndon words form a basis of $\Lie(a_1,\ldots,a_r)$, the free Lie algebra on $a_1,\ldots ,a_r$ over any commutative ring (\cite{Loth97}, Theorem 5.3.1). 

Let $J$ be a Lie algebra ideal in $\Lie(a_1,\ldots,a_r)$. Denote by $I$ the ideal generated by $J$ in $T(a_1,\ldots,a_r)$. Define a Lyndon word $\mathpzc{l}$ to be {\it $J$-standard} if $b(\mathpzc{l})$ cannot be written as a linear combination of strictly smaller (with respect to $\preceq$) Lyndon words modulo $J$. Let $S^JL$ be the set of $J$-standard Lyndon words. Similarly define the set $S^I$ of $I$-standard words with respect to the order $\preceq$. We will make use of the following results.
\begin{prop} \label{SL=L} {\bf (\cite{LaRam95}, Corollary 2.8)}
With $I,J$ as above we have $L \cap S^I = S^JL$.  
\end{prop}
\begin{prop}\label{SLbasis}{\bf (\cite{LaRam95}, Theorem 2.1)}
The set $S^JL$ is a basis of $\Lie(a_1,\ldots,a_r)/ J$ over $\Q$. 
\end{prop}

We restrict our attention to $J=(R)$ where $R$ has the form described above. Let $S_RL$ be the set of all Lyndon words which does not contain $a_1a_2$ consecutively in that order. From the Diamond Lemma, $S^I$ is the same as the $R$-irreducible monomials which are the monomials not of the form $Aa_1a_2B$. It follows that $S^{(R)}L=S^I \cap L = S_RL$.  We prove the following theorem.
\begin{theorem}\label{Liebasis}
Suppose $\RR$ is a localization of $\Z$. Then the images of $b(S_RL)$ form a basis of $L(V,R)$.
\end{theorem}

\begin{proof}
Proposition \ref{SLbasis} implies the Theorem for $\RR=\Q$. We use it to obtain the result for $\RR$, a localization of $\Z$. First check that the images must indeed generate the Lie algebra $L(V,R)$. Any Lyndon word outside $S^{(R)}L$ can be expressed using lesser elements and continuing in this way we must end up with elements in $S^{(R)}L$. Therefore $S^{(R)}L$ spans $L(V,R)$ and as $S^{(R)}L=S_RL$ the latter also spans. 

It remains to check that $S_RL$ is linearly independent. From Proposition \ref{SLbasis}, $S_RL$ is linearly independent over $\Q$ the fraction field of $\RR$. It follows that it must be linearly independent over $\RR$.
\end{proof}
\end{mysubsection}


\section{Examples of quadratic algebras coming from manifolds}\label{quadcoh}

In this section we prove that cohomology algebras of $(n-1)$-connected $(2n+1)$-manifolds with field coefficients are all quadratic algebras. Contrast this with \cite{BaBa15_unpub} where it is shown that the cohomology algebras of $(n-1)$-connected $2n$-manifolds are quadratic whenever the $n^{th}$-Betti number is $\geq 1$. We first consider a general kind of graded commutative algebra of which the cohomology algebras of $(n-1)$-connected $(2n+1)$-manifolds are an example, and then verify a simple criterion which makes these algebras quadratic.

Let $A$ be a graded commutative algebra over the field $k$ (assume that all of $A$ are in non-negative grading, that is $A_{<0} =0$). Suppose that $A$ contains a vector space $V$ and an element $z\in A_m$ such that the only non-trivial multiplication in positive grading of $A$ occurs when two elements of $V$ in homogeneous degree $r$ and $m-r$ multiply to give a multiple of $z$. Thus, we have a filtration $k\{1\}\subset k\oplus V \subset k\oplus V \oplus k\{z\} =A$ of $A$ as graded vector spaces, such that the multiplication rules are given by 
$$1\cdot v= v,~ 1\cdot z=z,~ v \cdot z=0~ \forall~ v\in V,$$
and there is a bilinear form $\phi : V\otimes V \to k$ such that 
$$v_1\cdot v_2 = \phi(v_1,v_2)z.$$ 
Note that such algebras are classified by $V$ and $\phi$, and we will henceforth call these {\it form algebras}. The bilinear form $\phi$ is graded-symmetric, that is, 
$$\phi(v_1,v_2)= (-1)^{|v_1||v_2|}\phi(v_2,v_1).$$
We say that such a graded-symmetric bilinear form is {\it non-degenerate} if it induces an isomorphism between $V$ and $V^\ast$ given by $v\mapsto \beta_\phi(v)$ where $\beta_\phi(v)(w)= \phi(w,v)$. We note the Proposition below whose proof is the same as that of \cite[Theorem 3.2]{MiHu73}
\begin{prop}\label{bil-spl}
Suppose that $W$ is a graded subspace of $V$ on which the restriction of $\phi$ is non-degenerate. Then $V\cong W\oplus W^\perp$. 
\end{prop}
We use the notation $H^s$ for the form whose matrix is $\begin{bmatrix} 0 & 1 \\ 1& 0\end{bmatrix}$, and $H^{a.s}$ for the form whose matrix is  $\begin{bmatrix} 0 & 1 \\ -1& 0\end{bmatrix}$. 

We let $R=\Ker( \phi)$ and ask the question whether the associated form algebra $A$ is the quadratic algebra $A(V,R)$. We  prove 
\begin{prop}\label{bil-quad}
If either $H^s \subset V$ or $H^{a.s}\subset V$, the algebra $A$ is isomorphic to the quadratic algebra $A(V,R)$. 
\end{prop}

\begin{proof}
If either $H^s \subset V$ or $H^{a.s}\subset V$, there are  $v_1$ and $v_2$ $\in V$ such that 
$$\phi(v_1,v_1)=0 = \phi(v_2,v_2),~\phi(v_1,v_2)=1.$$
We may grade $A(V,R)$ by weight by declaring the weight of the elements of $V$ as $1$. We obtain an algebra map 
$$\tau : A(V,R) \to A$$ 
by sending $V$ to $V$ by identity. Now this induces a linear isomorphism  
$$A(V,R)^{\leq 2} \to A$$
where $A(V,R)^{\leq 2}$ consists of the elements in weight grading $\leq 2$. In order to show that $\tau$ is an isomorphism we need to check  that $A(V,R)= A(V,R)^{\leq 2}$, that is, there are no non-trivial elements of weight $\geq 3$ in $A(V,R)$. It suffices to show that elements of the form $w_1\otimes w_2\otimes w_3$ map to $0$ in $A(V,R)$. 

If $w_2\otimes w_3$ is non-zero in $A(V,R)$ we have that $w_2\otimes w_3$ is some non-trivial multiple of $v_1\otimes v_2$. Thus, we may assume $w_2=v_1$ and $w_3=v_2$. Applying the same argument for $w_1\otimes v_1$ we see that it is enough to show that $v_1\otimes v_2 \otimes v_2$ is $0$. This is now clear as $\phi(v_2,v_2)=0$. 
\end{proof}

\begin{exam}
We observe that the above condition is necessary. Suppose we have a basis $v_1,\cdots, v_k$ of $V$ and $\phi$ is given by 
$$\phi(v_1,v_1)=1, ~~ \phi(v_i,v_j)=0 ~\mbox{if } (i,j)\neq (1,1).$$
Then, $R=\Ker(\phi)$ has a basis given by $v_i\otimes v_j$ where $(i,j)\neq (1,1)$. Note that $R\otimes V + V\otimes R$ does not contain the element  $v_1\otimes v_1\otimes v_1$ and so the algebra $A(V,R)$ has non-trivial elements in weight $\geq 3$. It follows that the associated algebra $A$ is not a quadratic algebra. 
\end{exam}


We apply Proposition \ref{bil-quad} to the cohomology ring of an $(n-1)$-connected $(2n+1)$-manifold $M$ ($n\geq 2$). From Poincar\'e duality, the homology of $M$ has the form 
\begin{myeq} \label{Mhom}
H_i(M)=\left \{\begin{array}{rl} 
                \Z &\mbox{if}~i=0,2n+1 \\
                \Z^r \oplus G &\mbox{if}~i=n  \\
               \Z^r                &\mbox{if} ~i=n+1           \\
0    &\mbox{otherwise},                   
\end{array}\right.
\end{myeq}
and the cohomology
\begin{myeq}\label{Mcoh}
H^i(M)=\left \{\begin{array}{rl} 
                \Z &\mbox{if}~i=0,2n+1 \\
                \Z^r&\mbox{if}~i=n  \\
               \Z^r    \oplus G              &\mbox{if} ~i=n+1           \\
0    &\mbox{otherwise},                   
\end{array}\right.
\end{myeq}
for some finite abelian group $G$. We note that the multiplication is commutative.

\begin{prop}\label{quadratic}
For any field $k$, $H^*(M;k)$ is a quadratic algebra. 
\end{prop}

\begin{proof}
With $k$ coefficients we have 
$$H^i(M;k)=\left \{\begin{array}{rl} 
                k &\mbox{if}~i=0,2n+1 \\
                k^s&\mbox{if}~i=n,n+1  \\
0    &\mbox{otherwise},                   
\end{array}\right.$$  
where $s$ equals $r$ plus the dimension of $G\otimes \Z/p$ where $p=\mathit{char}(k)$. 

We note that $H^\ast(M;k)$ is a form algebra with $V= H^n(M;k)\oplus H^{n+1}(M;k)$ and some choice of generator $z$ of $H^{2n+1}(M;k) \cong k$. Poincar\'{e} duality implies the multiplication $$H^n(M;k)\otimes H^{n+1}(M;k) \to    H^{2n+1}(M;k) \cong k$$
actually writes $H^{n+1}(M;k)$ as the dual vector space of $H^n(M;k)$. In the case $s\neq 0$, write down a basis $(v_1,\cdots, v_s)$ for $H^n(M;k)$ and a dual basis $(w_1,\cdots, w_s)$  for $H^{n+1}(M;k)$. The subspace $kv_1 + kw_1$ in $V$ induces an inclusion $H^s\to V$. Now the result follows from Proposition \ref{bil-quad}.   
\end{proof}

\begin{notation}
We fix a choice of graded modules $V_M= \Z^r \oplus \Z^r \subset    H_n(M)\oplus H_{n+1}(M)$ and $W_M=  \Z^r \oplus \Z^r \subset H^n(M)\oplus H^{n+1}(M)$ such that $W_M \stackrel{\cong}{\to} \Hom(V_M,\Z)$ in the universal coefficient theorem. We fix an orientation class $[M]\in H_{n+1}(M)$ and note that the multiplication in $H^\ast(M)$ induces a bilinear form $\Phi_M : W_M\otimes W_M \to \Z$ which is non-degenerate by Poincar\'{e} duality. Let $R_M = \Ker(\Phi_M)$. 
\end{notation}

With this notation we note from the proofs of Propostion \ref{quadratic} and Proposition \ref{bil-quad} that 

\begin{cor} \label{quad-man}
If $\mathit{char}(k)$ does not divide $|G|$, $H^\ast(M;k) \cong A(W_M\otimes k, R_M\otimes k)$. 
\end{cor}

The diagonal map on the chain level induces $\Delta:H_*(M)\to H_*(M\times M)$. This does not  induce a coalgebra structure on $H_\ast (M)$ as the  target is not quite $H_*(M)\otimes H_*(M)$. However, if we change coefficients so that the homology groups are torsion free, then $H_*(M)$ is a coalgebra. We list the primes $p$ for which  $G\otimes \Z/p \neq 0$ as $\Sigma_M = \{p_1,\cdots, p_k\}$ and write $D_M=\Z[\frac{1}{p_1},\cdots, \frac{1}{p_k}] \subset \Q$. Then, we have that $H_\ast (M;D_M)$ is torsion free, so that $H_\ast(M;D_M)$ is a coalgebra. We write $$\rho_M\subset V_M\otimes V_M\cong \Hom(W_M\otimes W_M,\Z)$$ 
as those linear maps which are zero on $R_M$. With these notations we observe    
\begin{prop}\label{quad-coalg}
As a coalgebra over the ring $D_M$ we have the indentification,  
$$H_*(M;D_M)\cong C(V_M, \rho_M).$$
\end{prop}

\begin{proof}
As the homology with $D_M$-coefficients is torsion free, the coalgebra structure  is dual to the ring structure on cohomology, and hence is the linear dual of $A(W_M,R_M)$. It follows that this is the quadratic coalgebra $C(V_M, \rho_M)$.
\end{proof}

\begin{notation}\label{bas-not}
Note that the bilinear form $\Phi_M$ writes the degree $n$ part of $W_M$ as the dual of the degree $n+1$ part of $W_M$. We write down a basis $w_1,\cdots, w_r$ of the degree $n$ part and the dual basis $w_1',\cdots, w_r'$ of the degree $n+1$ part. We note that $R_M$ is spanned by 
 $$\{w_i\otimes w_j, ~ w_i'\otimes w_j'~ \mbox{for all}~ i,j,~w_i \otimes w_i' - w_i' \otimes w_i~\mbox{for all}~i  , w_i\otimes w_j',~w_i'\otimes  w_j~\mbox{for}~i\neq j\}.$$   
We write the basis of $V_M$ dual to the above basis as $\{v_1\cdots, v_r, v_1',\cdots,v_r'\}$. The module $\rho_M\subset V_M\otimes V_M$ is one dimensional and we may read off a generator of it from the expression above as $\sigma_M=\sum_i v_i\otimes v_i'-v_i'\otimes v_i$.
\end{notation}
 We now apply Proposition \ref{Kos2} to deduce

\begin{prop}\label{man-Kos}
Let $k$ be a quotient field of $D_M$.  Then, the quadratic algebra $A(V_M\otimes k,\rho_M\otimes k)$  is Koszul.
\end{prop}

\section{Computation of homotopy groups in the torsion free case}\label{conn}

In this section we compute the homotopy groups of any closed $(n-1)$-connected $(2n+1)$-manifold $M$ with homology as in \eqref{Mhom} after inverting the primes in $\Sigma_M$, that is those primes $p$ for which $G\otimes \Z/p \neq 0$. In this case we are able to execute the method in \cite{BaBa15_unpub} and obtain an expression of the homotopy groups as a direct sum of the homotopy groups of spheres. We directly observe that when $r=0$ the manifold $M$ is homotopy equivalent to $S^{2n+1}$ after inverting the primes in $\Sigma_M$, so that we assume that $r\geq 1$ below. 

The technique used in computing homotopy groups starts with computation of the homology of the loop space $\Omega M$ as an associative algebra with respect to the Pontrjagin product. We express $H_*(\Omega M)$ as the universal enveloping algebra of a certain Lie algebra $\LL_r^u(M)$. The relevant Lie algebra possesses a countable basis $l_1,l_2,\ldots$ which correspond to mapping spheres into $M$ via iterated Whitehead products. We apply the Poincar\'{e}-Birkhoff-Witt Theorem (over the ring $D_M$) to finish off the computation. 

We start by computing the Hurewicz homomorphism $\pi_k M \stackrel{h}{\to} H_k M$ for $(n-1)$-connected $2n+1$-manifolds. For $k<n-1$, we have $\pi_k M =0$ and by the Hurewicz theorem we have that $\pi_n M \cong H_n M$. In degree $n+1$, we prove 
\begin{prop} \label{Mhur} 
The map $\pi_{n+1} M \to H_{n+1}M$ is surjective. 
\end{prop} 

\begin{proof} 
Since $M$ is $(n-1)$-connected and $\pi_n M \cong H_n M \cong \Z^r \oplus G$, we may form a map 
$$\tau: Y:= \vee_r S^n \vee M(G,n) \to M,$$
which induces an isomorphism on $\pi_n$. Here $M(G,n)$ is the Moore space in degree $n$ defined by $\tilde{H}_i (M)$ $=$ $G$ for $i=n$, and, $=$ $0$ for $i\neq n$. Let $C\tau$ denote the mapping cone of $\tau$. Note that $Y$ and $M$ are both $(n-1)$-connected, so from the Blakers-Massey Theorem \cite{BlM53} it follows that  $\pi_t(M,Y) \cong \pi_t(C\tau)$ for $t\leq 2n-1$.  The Hurewicz homomorphism induces the commutative diagram
$$\xymatrix{  \pi_{n+1}Y \ar[r] \ar[d] & \pi_{n+1}M \ar[r] \ar[d] & \pi_{n+1}(M,Y) \ar[r] \ar[d] & \pi_n Y \ar[r]^{\cong} \ar[d]^\cong & \pi_nM \ar[r] \ar[d]^\cong & 0 (=\pi_n(M,Y)) \\
H_{n+1}Y \ar[r] & H_{n+1}M  \ar[r]^\cong & H_{n+1}(C\tau)   \ar[r] & H_nY \ar[r]^\cong & H_nM \ar[r] & 0( = H_n(C\tau) ) }$$
whose rows are exact. From Hurewicz theorem we note that $\pi_{n+1}(M,Y) \to H_{n+1} (C\tau)$ is an isomorphism. Note also that $H_{n+1} Y =0$ from which it follows that $H_{n+1}(M) \to H_{n+1}(C\tau)$ is an isomorphism. As $\pi_{n+1}M\to \pi_{n+1}(M,Y)$ is surjective, it follows that $\pi_{n+1} M \to H_{n+1}M$ is surjective. 
\end{proof} 

\begin{notation}\label{map-gen}
Proposition \ref{Mhur} implies that the $v_i$ and $v_i'$ of Notation \ref{bas-not} are in the image of the Hurewicz map. We fix maps $\nu_i: S^n\to M$ and $\nu_i' : S^{n+1} \to M$ whose Hurewicz images are $v_i$ and $v_i'$ respectively. 
\end{notation}

Our next task is to compute the homology of the loop space. Recall that the homology of $M$ over the ring $D_M$ is given by the formula 
$$H_i(M)\cong\left \{\begin{array}{rl} 
                D_M &\mbox{if}~i=0,2n+1 \\
                D_M^r &\mbox{if}~i=n,n+1  \\
                0    &\mbox{otherwise}.                   
\end{array}\right.$$  
As in Notation \ref{bas-not}, we have a basis $v_1,\cdots, v_r$ of $H_n( M;D_M)$ and a basis $v_1',\cdots, v_r,$ of $H_{n+1}(M;D_M)$ so that $\rho_M$ is one-dimensional generated by $\sum_i v_i\otimes v_i' - v_i'\otimes v_i$. In $H^\ast(M;D_M)$ the only non-trivial products of positive dimensional classes are given by the intersection form. Therefore the module of indecomposables in $H^\ast(M;D_M)$ is given by $\A(M)= H^n(M;D_M)\oplus H^{n+1}(M;D_M)$. With respect to the basis $w_1,\cdots, w_r, w_1'\cdots, w_r'$ of $\A(M)$ and the fixed choice of orientation class $[M]$, we have 
$$\langle w_i w_j, [M]\rangle =0 = \langle w_i'w_j', [M]\rangle, ~~ \langle w_iw_j',[M]\rangle = \langle w_j'w_i, [M] \rangle = \delta_{ij}.$$ 

In \cite{BerBor15}, the homology of $\Omega M$ is computed for $(n-1)$-connected $M$ of dimension $\leq 3n-2$ with  $Rank( H^\ast (M ; \Z ) ) > 4$. Suppose $x_1,x_2,\cdots, x_s$ is a basis of the module of indecomposables $\A(M)$ over $\Q$ or some subring $R$ where this module is free and $c_{ij}=\langle x_ix_j,[M]\rangle$ for a fixed choice of $[M]$ . Consider the homology ring $H_\ast(\Omega M ; \Q)$ of the based loop space, equipped with the Pontrjagin product. This ring is freely generated as an associative algebra by classes $u_1 , \cdots, u_r $ whose homology suspensions are dual to the classes $x_1 , \cdots, x_r$ (in particular $|u_i | = |x_i | - 1$), modulo the single quadratic relation
$$\sum_{i,j}(-1)^{ |u_i |+1} c_ {ji} u_i u_j = 0,$$
The same argument works for any quotient field of $R$ where $R$ is a subring over which $\A(M)$ is free. In our case as above we choose $R=D_M$ and we have $u_1,\cdots, u_r, u_1',\cdots, u_r'$ as those whose homology suspensions are $v_1,\cdots, v_r, v_1',\cdots, v_r'$ respectively, then note that the expression above is precisely 
$$l_M:= \pm\sum_i u_i\otimes u_i' - u_i'\otimes u_i$$
where the sign $\pm$ is determined from the parity of $n$.  We note this result in Proposition below (cf. \cite[Theorem 1.1]{BerBor15}).
\begin{prop}\label{loopk}
For $k=\Q$ or a quotient field of $D_M$, there is an isomorphism of associative rings, 
$$H_*(\Omega M;k) \cong T_k(u_1,\cdots,u_r,u_1',\cdots.u_r' )/(l_M).$$ 
\end{prop}
\noindent This directly leads us to the following integral version which follows the same proof as \cite[Proposition 2.2]{BaBa16_unpub}.
\begin{prop}\label{manhomloop}
As associative rings,
 $$H_*(\Omega M;D_M) \cong T_{D_M}(u_1,\cdots,u_r,u_1',\cdots,u_r')/(l_M).$$ 
\end{prop}

We next use the computation of the homology of the loop space to split off the homotopy groups as a direct sum of homotopy groups of spheres. Note that the element $l_M = \sum_i u_i\otimes u_i' - u_i'\otimes u_i$ in $T_{D_M}(u_1,\cdots,u_r,u_1',\cdots,u_r')$ actually lies in the free Lie algebra generated by $u_i,u_i'$ which we denote by  $\Lie(u_1,\cdots,u_r,u_1',\cdots, u_r')$.  Consider the  Lie algebra $\LL(M)$ (over $D_M$) given by 
$$\frac{\Lie(u_1,\cdots,u_r,u_1',\cdots,u_r')}{(l_M)}$$
where $(l_M)$ denotes the Lie algebra ideal generated by $l_M$. Then $\LL(M)$ inherits a grading with $|u_i|$ $=$ $n-1$  and $|u_i'|$ $=$ $n$.  Denote by $\LL_w(M)$ the degree $w$ homogeneous elements of $\LL(M)$. From Proposition \ref{freemod} and Theorem \ref{Liebasis} we know that $\LL(M)$ is a free module and the Lyndon basis gives a basis of $\LL(M)$.

List the elements of the Lyndon basis in order as $l_1 < l_2 <\ldots$ and define the height of a basis element by $h_i= h(l_i)=w+1$ if $b(l_i) \in \LL_w(M)$. Then $h(l_i)\leq h(l_{i+1})$. Note that $b(l_i)$ represents an element of $\Lie(u_1,\cdots,u_r,u_1',\cdots, u_r')$ and is thus represented by an iterated Lie bracket of the $u_i$s and $u_j'$s. Use iterated Whitehead products of $\nu_i, \nu_i'$ to define maps $\lambda_i : S^{h_i}\rightarrow M$. 



\begin{theorem}\label{htpytors-free}
For primes $p$ that are not invertible in $D_M$ (that is, $p\notin \Sigma_M$), there is an isomorphism\footnote{Note that the right hand side is a finite direct sum for each $\pi_n(M)$.}
$$\pi_*(M) \cong \sum_{i\geq 1} \pi_* S^{h_i}$$ 
and the inclusion of each summand is given by $\lambda_i$.
\end{theorem}

\begin{proof}
Following \cite[Theorem 4.1]{BaBa15_unpub} write $S(t) = \prod_{i=1}^t \Omega S^{h_i}$. The maps $\Omega \lambda_i: \Omega S^{h_i} \rightarrow \Omega M$ for $i=1,\ldots,l$ can be multiplied using the $H$-space structure on $\Omega M$ to obtain a map from $S(l)$ to $\Omega M$. Use the model for $\Omega M$ given by the Moore loops so that the multiplication is strictly associative with a strict identity. Then the maps  $S(t) \rightarrow \Omega M$ and $S(t')\rightarrow \Omega M$ for $t\leq t'$ commute with the inclusion $S(t)\rightarrow S(t')$ by the basepoint on the last $t'-t$ factors.  Hence we obtain a map 
$$\Lambda: S := \mathit{hocolim}~ S(t) \rightarrow \Omega M$$   
We prove that  $\Lambda_*$ is an isomorphism on homology with $D_M$ coefficients. 

We know from Proposition \ref{manhomloop} that 
$$H_*(\Omega M;D_M) \cong T_{D_M}(u_1,\cdots,u_r,u_1',\cdots,u_r')/(l_M).$$
 Proposition \ref{univlie} implies that  this  is  the universal enveloping algebra of the Lie algebra $\LL(M)$. Hence, by the Poincar\'e-Birkhoff-Witt Theorem we have  
$$E_0H_*(\Omega M) \cong S(\LL(M))$$
where $E_0H_*(\Omega M)$ is the associated graded algebra of $H_*(\Omega M)$ with respect to the length filtration and $S(\LL(M))$ is the symmetric algebra on $\LL(M)$. 

The homology of $S$ is the algebra 
$$H_*(S;D_M) \cong T_{D_M} (c_{h_1 -1})\otimes T_{D_M}(c_{h_2-1})\otimes\cdots \cong D_M[c_{h_1-1},c_{h_2 -1},\cdots ]$$
where $c_{h_i -1}$ denotes the generator in $H_{h_i -1}(\Omega S^{h_i};D_M)$. Now each $c_{h_i-1}$ maps to the  Hurewicz image of $\Omega (\lambda_i)\in H_{h_i-1}(\Omega M;D_M)$.  Consider the composite 
$$\rho : \pi_n(M)\cong \pi_{n-1}(\Omega M) \stackrel{\mathit{Hur}}{\longrightarrow} H_{n-1}(\Omega M).$$
We know from \cite{Hil55}, Lemma 2.2 (also see \cite{Sam53}) that 
\begin{myeq}\label{hur}
\rho([a,b])=\pm (\rho(a)\rho(b) - (-1)^{|a||b|}\rho(b)\rho(a)).
\end{myeq}
 The map $\rho$ carries each $\nu_i$  to $u_i$ and $\nu_i'$ to $u_i'$. The element  $b(l_i)$ is mapped inside $H_*(\Omega M)$ to the element corresponding to the graded Lie algebra element (up to sign) by equation (\ref{hur}). Denote the Lie algebra element (ungraded) corresponding to $b(l_i)$ by the same notation. We readily discover that the difference of the graded and ungraded elements lie in  the algebra generated by terms of lower weight, i.e.,
\begin{myeq}
\label{rhobl}
\rho(b(l_i)) \equiv b(l_i)~~ \pmod{\mathit{lower~order~terms}}.
\end{myeq}
Therefore
$$E_0T_{D_M}(u_1,\cdots,u_r,u_1',\cdots,u_r')/(l_M) \cong D_M [b(l_1),b(l_2),\ldots] \cong  D_M [\rho(b(l_1)),\rho(b(l_2)),\ldots].$$
It follows that the monomials in $\rho(b(l_i))$, $i=1,2,\ldots$ form a basis of  $H_*(\Omega M) $. The map $\Lambda : S\to \Omega M$ maps $c_{h_i-1} \to \rho(b_i)$ and takes the product of elements in $D_M[c_{h_1-1},c_{h_2 -1},\ldots ]$ to the corresponding Pontrjagin product. It follows that $\Lambda_*$ is an isomorphism.

It follows that if $p\notin \Sigma_M$ then $\Lambda_*$ is an isomorphism with $\Z_{(p)}$-coefficients. Both $S$ and $\Omega M$ are $H$-spaces, and hence simple (that is $\pi_1$ is abelian and acts trivially on $\pi_n$ for $n\ge 2$). It follows that $\Lambda$ is a weak equivalence when localized at $p$. Hence  the result follows. 
\end{proof}

The proof of Theorem \ref{htpytors-free} implies a stronger result about the loop space of the manifold $M$. We denote the $D_M$-localization of $M$ by $M_\tau$. For a sequence of based spaces $Y_i$, we use the notation $\hat{\Pi}_{i\geq 0} Y_i$ for the homotopy colimit of finite products of $Y_i$.  

\begin{theorem} \label{loopdectors-free}
With notations as above, 
$$\Omega M_\tau \simeq \hat{\Pi}_{i\geq 0} \Omega S^{h_i}_\tau.$$
\end{theorem}

We now compute the number of copies of $S^k$ in the expression of Theorem \ref{htpytors-free} from the rational cohomology groups of $M$. Let
$$q_M(t)= 1 - rt^n -rt^{n+1} + t^{2n+1}$$
Then $\frac{1}{q_M(t)}$ is the generating series for $\Omega M$  \cite[Theorem 3.5.1]{LoVa12}, from the fact that $H_*(\Omega M ; D_M)$ is Koszul as an associative algebra \cite{BerBor15}. Let
$$\eta_m := \mbox{coefficient of $t^m$ in }log(q_M(t)).$$
We may now repeat the proof of  \cite[Theorem 3.8]{BaBa15_unpub} to deduce 
\begin{theorem}\label{htpyformtors-free}
 The number of groups $\pi_s S^m\otimes D_M$ in $\pi_s(M)\otimes D_M$ is 
$$l_{m-1}= - \sum_{j|m-1} \mu(j)\frac{\eta_{(m-1)/j}}{j}$$
where $\mu$ is the M\"obius function.
\end{theorem}

%
%

These computations have consequences in relation to exponents of homotopy groups of $(n-1)$-connected $(2n+1)$-manifolds. We first note that by a theorem of Miller \cite{Mil79} that these manifolds are all formal. Recall that simply connected, finite cell complexes either have finite dimensional rational homotopy groups or exponential growth of ranks of rational homotopy groups (cf. \cite[\S 33]{FHT01}). The former are called rationally elliptic while the latter are called rationally hyperbolic. We note that the $(n-1)$-connected manifolds of dimension $2n+1$ are rationally hyperbolic if and only if $r\geq 2$. For if $r\geq 2$, there are at least two of the generating $u_i$ and $u_i'$. Then one observes that after switching the ordering appropriately the word
$$u_1'u_2'u_1'u_1u_1'$$
is a Lyndon word in degree $>2(n+1) (=2\dim(M))$ which induces a non-trivial rational homotopy group in dimension $> 2\dim(M)$ implying that $M$ is rational hyperbolic. If $r=0$, the rational cohomology looks like $S^{2n+1}$ and if $r=1$, the rational cohomology looks like $S^n \times S^{n+1}$,  then formality implies that the manifolds are rationally elliptic.

There are many conjectures that lie in the dichotomy between rationally elliptic and hyperbolic spaces. One such is a conjecture by Moore \cite{NeSe81} (also \cite[pp.518]{FHT01}) which implies that  for a rationally hyperbolic space $X$, there are primes $p$ for which the homotopy groups do not have any exponent at $p$, that is, for any power $p^r$ there is an element $\alpha\in \pi_*(X)$ of order $p^r$. We verify the following version.
\begin{theorem}\label{Moorehcm}
If $p\notin \Sigma_M$ and $r\geq 2$, the homotopy groups of $M$ do not have any exponent at $p$. 
\end{theorem}   

\begin{proof}
We have noted above that these $M$ are rationally hyperbolic, so it follows that there are non-trivial rational homotopy groups in arbitrarily large dimensions. It follows that in the expression of Theorem \ref{htpyformtors-free}  for arbitrarily large $l$, $\pi_*S^l_{(p)}$ occurs as a summand of $\pi_*M_{(p)}$. Now we observe \cite{Gra69}  that any $p^s$ may occur as the order of an element in $\pi_*S^l$ for arbitrarily large $l$. \footnote{ This also follows from the fact that the same is true for the stable homotopy groups and these can be realized as $\pi_k^s\cong \pi_{k+l}S^l$  for $l>k+1$. Now torsion of order $p^s$ for any $s$ occurs in the image of the $J$-homomorphism \cite[Theorem 1.1.13]{Rav86}.}  
\end{proof}

In section \ref{loopdectors}, using a different method we verify that $\Omega (S^n \vee S^{n+1})$ is a retract of $\Omega M$ if $r\geq 2$ (Corollary \ref{rethyp}).  It follows from \cite{Hil55} that for such a $M$, $\pi_\ast M$ has summands $ \pi_\ast S^k$ for $k$ arbitrarily large, so they cannot have homotopy exponents at any prime $p$. 

\section{Loop space decompositions} \label{loopdectors}
In this section we provide a splitting of the loop space of a $(n-1)$-connected $(2n+1)$-manifold with homology described in \eqref{Mhom}, in the case $r\geq 1$. This splitting is similar to the loop space decompositions of $(n-1)$-connected $2n$-manifolds in \cite{BeTh14}. Recall from    Proposition \ref{Mhur} that the Hurewicz map in degree $n$ and $n+1$ is surjective. As constructed in the proof of Proposition \ref{Mhur} we have a map 
$$\tau : Y := \vee_r S^n \vee M(G,n) \to M$$
 We recall from Notation \ref{map-gen} that we may choose the restriction of $\tau$ on the $i^{th}$ sphere of the wedge to be $\nu_i$. We analogously use the $\nu_i'$ to construct a map 
$$\tau': Y':= \vee_r S^{n+1} \to M.$$
Let $Z:=\vee_{r-1} S^n \vee_{r-1}S^{n+1} \vee M(G,n)$ which we think of as a subspace of the wedge of $Y$ and $Y'$ leaving out the last sphere of $Y$ and the last sphere of $Y'$. We write $\kappa$ for the induced map $Z\to M$. Let $Q$ denote the mapping cone of $\kappa$. From our choice of $\kappa$, we have that $H^\ast(Q) \cong H^\ast(S^n \times S^{n+1})$. 

We note that  the composites 
$$\lambda: S^n \stackrel{\nu_r}{\to} M \to Q,~~ \lambda': S^{n+1} \stackrel{\nu_r'}{\to} M \to Q$$
 map onto the generators on $H_n(Q)$ and $H_{n+1}(Q)$ respectively. We easily compute as in \cite{BerBor15}  $H_\ast (\Omega Q) \cong \Z[u,v]$  with $|u|=n-1$ and $|v|=n$ so that the map
$$\Omega \lambda : \Omega S^n \to \Omega Q$$
sends the generator in $H_{n-1}(\Omega S^n)$ to $u$, and the map
$$\Omega \lambda' : \Omega S^{n+1} \to \Omega Q$$
sends the generator in $H_n(\Omega S^{n+1})$ to $v$. The composite 
$$\Omega S^n \times \Omega S^{n+1} \stackrel{\Omega \lambda \times \Omega \lambda'}{\to} \Omega Q \times \Omega Q \to \Omega Q$$ 
is clearly a homology isomorphism, and hence a weak equivalence. 

We write $F$ for the homotopy fibre of the map $M\to Q$, so that we have a fibration 
$$F\to M \to Q$$ 
and that the map $\kappa : Z \to M$ factors through $F$. Proceeding as in \cite{BeTh14} we prove two lemmas. The first is a loop space decomposition result $\Omega M \simeq \Omega Q \times \Omega F$. The second is an identification of the homotopy fibre $F\simeq \Omega Q \ltimes Z$, where the half-smash product $\ltimes$ is defined as 
$$X\ltimes Y := X_+\wedge Y \simeq X\times Y/X\times \ast.$$ 
 We start with the loop space decomposition result. 
\begin{lemma}\label{loopM}
As a space $\Omega M$ splits as a product
$$\Omega M \simeq \Omega F \times \Omega Q.$$
Further in the fibration $\Omega Q \to F \to M$ the inclusion of the fibre $\Omega Q \to F$ is null-homotopic. 
\end{lemma}  

\begin{proof}
We continue the fibration sequence $F\to M \to Q$ to obtain a fibration $\Omega F \to \Omega M \to \Omega Q$. It suffices to show that there is a right homotopy inverse to the map $\Omega M \to \Omega Q$. Write $V=\vee_r S^n \vee_r S^{n+1}$ viewed as a subspace of $Y\vee Y'$ so that $\tau$ and $\tau'$ induce a map $\tau_V: V \to M$. Consider the composite $\pi$
$$\Omega V \stackrel{\Omega \tau_V}{\longrightarrow} \Omega M  \to \Omega Q.$$
Now under the identification $\Omega Q \simeq \Omega S^n \times \Omega S^{n+1}$ we observe that $\pi\simeq \Omega p$ where $p$ is defined as the composite 
$$V \to S^n \vee S^{n+1} \to S^n \times S^{n+1},$$
the first map being the one which quotients out the factors common in $Z$ and $V$. As $V$ is a wedge of spheres, it is clear from the Hilton-Milnor Theorem that $\Omega p$ has a right inverse.
\end{proof}

Next we proceed towards proving that $F$ splits as a half-smash product of $\Omega Q$ and $Z$ as in \cite[Proposition 2.5]{BeTh14}. The first step involves computing the homology of $F$ from the Serre spectral sequence for the principal fibration $\Omega Q \to F \to M$. Note that from a result of Moore \cite{Moo56} that this is a spectral sequence of left $H_\ast \Omega Q$-modules.  

\begin{prop} \label{homF}
As a left $H_\ast(\Omega Q)$-module, the homology of $F$ is given by the formula 
$$\tilde{H}_\ast(F) \cong H_\ast(\Omega Q)\otimes \tilde{H}_\ast(Z).$$
\end{prop}

\begin{proof}
In the spectral sequence for $\Omega Q\to F \to M$ we have, 
$$E^2_{\ast,\ast} = H_\ast (M) \otimes H_\ast(\Omega Q) \implies H_\ast (F).$$
The tensor product decomposition of the $E^2$-term arises from the fact that $H_\ast (\Omega Q) \cong \Z[u,v]$ is free Abelian in each degree. For degree reasons the only non-trivial differentials are $d^n$ and $d^{n+1}$. We proceed to compute these two differentials to deduce the result. 

We write $H_\ast(M)$ as 
$$H_\ast(M) \cong \Z \oplus \tilde{H}_\ast(Z) \oplus \Z\{v_r,v_r',[M]\}$$
Note that the inclusion $Z\to M$ has  a lift to $F$, so that $\tilde{H}_\ast(Z)$ in the $0^{th}$-row of the $E^2$-page must survive to the $E^\infty$-page. From the $H_\ast(\Omega Q)$-module structure, it follows that all the differentials on $H_\ast(\Omega Q)\otimes \tilde{H}_\ast(Z)$ must be $0$.  It remains to compute the differentials on the classes $v_r$, $v_r'$ and $[M]$. For this note the commutative diagram of principal fibrations
$$\xymatrix{ \Omega Q \ar@{=}[r] \ar[d]    & \Omega Q \ar[d] \\ 
                          F  \ar[d] \ar[r]                          &   PQ   \ar[d] \\ 
                          M \ar[r]                                     & Q. }$$ 
We note that $v_r$, $v_r'$ and $[M]$ in $H_\ast M$ map respectively onto the generators in $H_nQ$, $H_{n+1} Q$ and $H_{2n+1}Q$. By comparing the two homology Serre spectral sequences, we deduce 
$$d^n(v_r)= u, ~~ d^{n+1}(v_r')=v,~~ d^n([M])= u\otimes v_r'.$$
Now in the spectral sequence we have $E^n = E^2$ and the above formula imply that 
$$E^{n+1}_{\ast, q} = \begin{cases} 
                                       \Z[v]            & \mbox{if} ~ q=0 \\  
                                    \Z[u,v]\otimes ( \Z\{v_1,\cdots, v_{r-1}\}\oplus G) & \mbox{if} ~ q=n \\  
                        \Z[u,v]\otimes ( \Z\{v_1',\cdots, v_{r-1}'\}\oplus G) \oplus \Z[v]\otimes \Z\{v_r'\} & \mbox{if} ~ q=n+1 \\  
                                      0  & \mbox{otherwise.}
\end{cases}$$ 
In the $(n+1)$-page, the differential $d^{n+1}$ sends the factor $\Z[v]\otimes \Z[v_r']$ onto $(v)\subset \Z[v]$. Therefore the $E^{n+2}$ is 
$$E^{n+2}_{\ast, q} = \begin{cases} 
                                       \Z            & \mbox{if} ~ q=0 \\  
                                    \Z[u,v]\otimes ( \Z\{v_1,\cdots, v_{r-1}\}\oplus G) & \mbox{if} ~ q=n \\  
                        \Z[u,v]\otimes ( \Z\{v_1',\cdots, v_{r-1}'\}\oplus G) & \mbox{if} ~ q=n+1 \\  
                                      0  & \mbox{otherwise.}
\end{cases}$$
There are no more non-trivial differentials, so that $E^{n+2}=E^\infty$. Note that this implies that the $E^\infty$-page is precisely $H_\ast(\Omega Q)\otimes \tilde{H}_\ast(Z)$. In the $E^\infty$-page, the only non-zero lines are the vertical $n$-line and the vertical $(n+1)$-line, and  all the possible torsion lies along the vertical $n$-line. Thus the extension problem in each case looks like 
$$0\to  \Z^k \otimes (\Z^{r-1}\oplus G) \to H_\ast F \to \Z^{k'} \to 0$$
which are clearly all trivial.  Hence the result follows.  
\end{proof}

We now use the computation of Proposition \ref{homF} to deduce our second decomposition result 
\begin{lemma}\label{decF}
There is a homotopy equivalence 
$$F\simeq \Omega Q \ltimes Z .$$
\end{lemma}

\begin{proof}
We fix a lift of $\kappa:Z \to M$ to $F$ and call if $j$. Also denote by $\lambda : \Omega Q \times F \to F$ the action of $\Omega Q$ on $F$. We have the composite 
$$\Omega Q \times Z \stackrel{id \times j}{\longrightarrow} \Omega Q\times F \stackrel{\lambda}{\rightarrow} F.$$
The restriction of the map to the subspace $\Omega Q\times \ast$ is the inclusion of the fibre of $\Omega Q \to F \to M$ which by Proposition \ref{loopM} is null-homotopic. Hence we obtain a map $\Omega Q \ltimes Z\to F$ which induces a homology isomorphism by Proposition \ref{homF}.  The result follows.
\end{proof}

We now apply the results proved above to deduce the loop space decomposition for $M$. 
\begin{theorem}\label{loopdecM}
Suppose $M$ is a $(n-1)$-connected $(2n+1)$-manifold with homology as in \eqref{Mhom} satisfying $r\geq 1$. Then we have a homotopy equivalence
$$\Omega M \simeq \Omega S^n \times \Omega S^{n+1} \times \Omega (Z\vee (Z\wedge \Omega (S^n \times S^{n+1})))$$
where $Z \simeq \vee_{r-1}S^n \vee_{r-1}S^{n+1} \vee M(G,n)$. 
\end{theorem}

\begin{proof}
From Lemma \ref{loopM} we have 
$$\Omega M \simeq \Omega F \times \Omega Q \simeq \Omega F \times \Omega S^n \times \Omega S^{n+1}.$$
From Lemma \ref{decF} we have 
$$F \simeq \Omega Q \ltimes Z \simeq \Omega Q_+ \wedge Z$$
with $Z$ as required. We note that $n\geq 2$ so that $Z$ is a suspension so that $X \ltimes Z \simeq X \wedge Z \vee Z$ for any based space $X$. This completes the proof.
\end{proof}

We note that in the expression above, $\Omega Z$ is a retract of $\Omega M$. If $M$ is rationally hyperbolic (that is $r\geq 2$) we note that $S^n \vee S^{n+1}$ is a wedge summand of $Z$, so that we have the corollary 
\begin{cor} \label{rethyp}
 Suppose $M$ is a $(n-1)$-connected $(2n+1)$-manifold with homology as in \eqref{Mhom} satisfying $r\geq 2$. Then, $\Omega (S^n \vee S^{n+1})$ is a retract of $\Omega M$. 
\end{cor}

\vspace*{0.5cm}


\begin{thebibliography}{10}

\bibitem{Ada60}
{\sc J.~F. Adams}, {\em On the non-existence of elements of {H}opf invariant
  one}, Ann. of Math. (2), 72 (1960), pp.~20--104.

\bibitem{Ada72}
{\sc J.~F. Adams}, {\em Algebraic topology---a student's guide}, Cambridge
  University Press, London-New York, 1972.
\newblock London Mathematical Society Lecture Note Series, No. 4.

\bibitem{BaBa16_unpub}
{\sc S.~Basu and S.~Basu}, {\em Homotopy groups of certain highly connected
  manifolds via loop space homology}.
\newblock available at \url{https://arxiv.org/abs/1601.04413}.

\bibitem{BaBa15_unpub}
\leavevmode\vrule height 2pt depth -1.6pt width 23pt, {\em Homotopy groups of
  highly connected manifolds}.
\newblock available at \url{http://arxiv.org/abs/1510.05195}.

\bibitem{BaBa15}
\leavevmode\vrule height 2pt depth -1.6pt width 23pt, {\em Homotopy groups and
  periodic geodesics of closed 4-manifolds}, Internat. J. Math., 26 (2015),
  pp.~1550059, 34.

\bibitem{BeTh14}
{\sc P.~Beben and S.~Theriault}, {\em The loop space homotopy type of
  simply-connected four-manifolds and their generalizations}, Adv. Math., 262
  (2014), pp.~213--238.

\bibitem{BeWu15}
{\sc P.~Beben and J.~Wu}, {\em The homotopy type of a {P}oincar\'e duality
  complex after looping}, Proc. Edinb. Math. Soc. (2), 58 (2015), pp.~581--616.

\bibitem{BerBor15}
{\sc A.~Berglund and K.~B\"orjeson}, {\em Free loop space homology of highly
  connected manifolds}, Forum Math., 29 (2017), pp.~201--228.

\bibitem{Berg78}
{\sc G.~M. Bergman}, {\em The diamond lemma for ring theory}, Adv. in Math., 29
  (1978), pp.~178--218.

\bibitem{BlM53}
{\sc A.~L. Blakers and W.~S. Massey}, {\em The homotopy groups of a triad.
  {III}}, Ann. of Math. (2), 58 (1953), pp.~409--417.

\bibitem{Car58}
{\sc P.~Cartier}, {\em Remarques sur le th\'eor\`eme de {B}irkhoff-{W}itt},
  Ann. Scuola Norm. Sup. Pisa (3), 12 (1958), pp.~1--4.

\bibitem{CMN79b}
{\sc F.~R. Cohen, J.~C. Moore, and J.~A. Neisendorfer}, {\em The double
  suspension and exponents of the homotopy groups of spheres}, Ann. of Math.
  (2), 110 (1979), pp.~549--565.

\bibitem{CMN79}
\leavevmode\vrule height 2pt depth -1.6pt width 23pt, {\em Torsion in homotopy
  groups}, Ann. of Math. (2), 109 (1979), pp.~121--168.

\bibitem{Cohn63}
{\sc P.~M. Cohn}, {\em A remark on the {B}irkhoff-{W}itt theorem}, J. London
  Math. Soc., 38 (1963), pp.~197--203.

\bibitem{DuLi05}
{\sc H.~Duan and C.~Liang}, {\em Circle bundles over 4-manifolds}, Arch. Math.
  (Basel), 85 (2005), pp.~278--282.

\bibitem{FHT01}
{\sc Y.~F{\'e}lix, S.~Halperin, and J.-C. Thomas}, {\em Rational homotopy
  theory}, vol.~205 of Graduate Texts in Mathematics, Springer-Verlag, New
  York, 2001.

\bibitem{Gra69}
{\sc B.~Gray}, {\em On the sphere of origin of infinite families in the
  homotopy groups of spheres}, Topology, 8 (1969), pp.~219--232.

\bibitem{Hil55}
{\sc P.~J. Hilton}, {\em On the homotopy groups of the union of spheres}, J.
  London Math. Soc., 30 (1955), pp.~154--172.

\bibitem{Jam57Ann}
{\sc I.~M. James}, {\em On the suspension sequence}, Ann. of Math. (2), 65
  (1957), pp.~74--107.

\bibitem{LaRam95}
{\sc P.~Lalonde and A.~Ram}, {\em Standard {L}yndon bases of {L}ie algebras and
  enveloping algebras}, Trans. Amer. Math. Soc., 347 (1995), pp.~1821--1830.

\bibitem{Laz54}
{\sc M.~Lazard}, {\em Sur les alg\`ebres enveloppantes universelles de
  certaines alg\`ebres de {L}ie}, Publ. Sci. Univ. Alger. S\'er. A., 1 (1954),
  pp.~281--294 (1955).

\bibitem{LoVa12}
{\sc J.-L. Loday and B.~Vallette}, {\em Algebraic operads}, vol.~346 of
  Grundlehren der Mathematischen Wissenschaften [Fundamental Principles of
  Mathematical Sciences], Springer, Heidelberg, 2012.

\bibitem{Loth97}
{\sc M.~Lothaire}, {\em Combinatorics on words}, Cambridge Mathematical
  Library, Cambridge University Press, Cambridge, 1997.
\newblock With a foreword by Roger Lyndon and a preface by Dominique Perrin,
  Corrected reprint of the 1983 original, with a new preface by Perrin.

\bibitem{McGN85}
{\sc C.~A. McGibbon and J.~A. Neisendorfer}, {\em Various applications of
  {H}aynes {M}iller's theorem}, in Conference on algebraic topology in honor of
  {P}eter {H}ilton ({S}aint {J}ohn's, {N}fld., 1983), vol.~37 of Contemp.
  Math., Amer. Math. Soc., Providence, RI, 1985, pp.~91--98.

\bibitem{Mill84}
{\sc H.~Miller}, {\em The {S}ullivan conjecture on maps from classifying
  spaces}, Ann. of Math. (2), 120 (1984), pp.~39--87.

\bibitem{Mil79}
{\sc T.~J. Miller}, {\em On the formality of {$(k-1)$}-connected compact
  manifolds of dimension less than or equal to {$4k-2$}}, Illinois J. Math., 23
  (1979), pp.~253--258.

\bibitem{MiHu73}
{\sc J.~Milnor and D.~Husemoller}, {\em Symmetric bilinear forms},
  Springer-Verlag, New York-Heidelberg, 1973.
\newblock Ergebnisse der Mathematik und ihrer Grenzgebiete, Band 73.

\bibitem{Moo56}
{\sc J.~C. Moore}, {\em The double suspension and {$p$}-primary components of
  the homotopy groups of spheres}, Bol. Soc. Mat. Mexicana (2), 1 (1956),
  pp.~28--37.

\bibitem{Nei81}
{\sc J.~A. Neisendorfer}, {\em {$3$}-primary exponents}, Math. Proc. Cambridge
  Philos. Soc., 90 (1981), pp.~63--83.

\bibitem{NeSe81}
{\sc J.~A. Neisendorfer and P.~S. Selick}, {\em Some examples of spaces with or
  without exponents}, in Current trends in algebraic topology, {P}art 1
  ({L}ondon, {O}nt., 1981), vol.~2 of CMS Conf. Proc., Amer. Math. Soc.,
  Providence, R.I., 1982, pp.~343--357.

\bibitem{PolPos05}
{\sc A.~Polishchuk and L.~Positselski}, {\em Quadratic algebras}, vol.~37 of
  University Lecture Series, American Mathematical Society, Providence, RI,
  2005.

\bibitem{Rav86}
{\sc D.~C. Ravenel}, {\em Complex cobordism and stable homotopy groups of
  spheres}, vol.~121 of Pure and Applied Mathematics, Academic Press, Inc.,
  Orlando, FL, 1986.

\bibitem{Sam53}
{\sc H.~Samelson}, {\em Classifying spaces and spectral sequences}, American J.
  of Math., 75 (1953), pp.~744--752.

\bibitem{Sel78}
{\sc P.~Selick}, {\em Odd primary torsion in {$\pi _{k}(S^{3})$}}, Topology, 17
  (1978), pp.~407--412.

\bibitem{Ser51}
{\sc J.-P. Serre}, {\em Homologie singuli\`ere des espaces fibr\'es.
  {A}pplications}, Ann. of Math. (2), 54 (1951), pp.~425--505.

\bibitem{Sma62}
{\sc S.~Smale}, {\em On the structure of {$5$}-manifolds}, Ann. of Math. (2),
  75 (1962), pp.~38--46.

\bibitem{Tod56}
{\sc H.~Toda}, {\em On the double suspension {$E^2$}}, J. Inst. Polytech. Osaka
  City Univ. Ser. A., 7 (1956), pp.~103--145.

\bibitem{Tod62}
\leavevmode\vrule height 2pt depth -1.6pt width 23pt, {\em Composition methods
  in homotopy groups of spheres}, Annals of Mathematics Studies, No. 49,
  Princeton University Press, Princeton, N.J., 1962.

\end{thebibliography}
\end{document}